\documentclass{amsart}

\usepackage{amsmath}
\usepackage{amsfonts}
\usepackage{amssymb}
\usepackage{graphicx}
\usepackage{hyperref}
\usepackage{mathrsfs}
\usepackage{pb-diagram}
\usepackage{epstopdf}
\usepackage{stmaryrd}
\usepackage{amsmath,amsfonts,amsthm,enumerate,amscd,latexsym}
\usepackage{curves}
\usepackage{bbm}
\usepackage[mathscr]{eucal}
\usepackage{caption}
\usepackage{subcaption}
\usepackage{tikz}
\newtheorem{thm}{Theorem}[section]
\newtheorem{lem}[thm]{Lemma}

\newtheorem{prop}[thm]{Proposition}
\newtheorem{defn}[thm]{Definition}
\newtheorem{example}[thm]{Example}
\newtheorem{remark}[thm]{Remark}
\newtheorem{conj}[thm]{Conjecture}

\numberwithin{equation}{section}

\def\bZ{\mathbb{Z}}

\def\bR{\mathbb{R}}
\def\bC{\mathbb{C}}
\def\bP{\mathbb{P}}

\subjclass[2010]{53D37, 14J33 (primary) and 53D40, 53D45, 14N35 (secondary)}


\begin{document}

\title[Quasimap SYZ for toric CY]{Quasimap SYZ for toric Calabi-Yau manifolds}
\author[Kwokwai Chan]{Kwokwai Chan}
\address{Department of Mathematics\\ The Chinese University of Hong Kong\\ Shatin\\ Hong Kong}
\email{kwchan@math.cuhk.edu.hk}

\date{\today}

\begin{abstract}
In this note, we study the SYZ mirror construction for a toric Calabi-Yau manifold using instanton corrections coming from Woodward's quasimap Floer theory \cite{Woodward11} instead of Fukaya-Oh-Ohta-Ono's Lagrangian Floer theory \cite{FOOO-book, FOOO-toricI, FOOO-toricII, FOOO-toricIII}. We show that the resulting SYZ mirror coincides with the one written down via physical means \cite{Leung-Vafa98, Hori-Vafa00, HIV00} (as expected).
\end{abstract}

\maketitle

%\tableofcontents

\section{Introduction}

The famous SYZ conjecture, proposed by Strominger, Yau and Zaslow \cite{SYZ96} in 1996, claims that mirror symmetry can be explained as a duality between Lagrangian torus fibrations. This suggests a nice geometric construction of the mirror for a given Calabi-Yau manifold $X$, namely, a mirror $\check{X}$ is given by the total space of the fiberwise dual of a Lagrangian torus fibration $\rho: X \to B$ on $X$. However, this construction cannot be right in general because usually a Lagrangian torus fibration admits singular fibers (which account for instanton corrections that make mirror symmetry interesting and powerful in applications to enumerative problems).

So we can only perform the duality over the smooth fibers. Since there is an integral affine structure (with singularities) on the base $B$ induced from $\rho: X \to B$, we have a natural complex structure $\check{J}_0$ on the total space of the dual fibration. But $\check{J}_0$ is defined only on an open dense subset of the mirror (as we have removed the singular fibers) and it cannot be extended any further due to nontrivial monodromy of the integral affine structure around the discriminant locus. Here comes the most important idea in the SYZ proposal \cite{SYZ96}: one has to modify the mirror complex structure (so-called semi-flat complex structure) by {\em instanton corrections} coming from holomorphic disks in $X$ with boundaries on smooth Lagrangian torus fibers of $\rho: X \to B$.

In terms of Lagrangian Floer theory, this means that the mirror $\check{X}$ should be given as a moduli space of pairs $(L,\nabla)$ consisting of a Lagrangian torus fiber $L$ and a flat $U(1)$ connection $\nabla$ on $L$, where the equivalence relation is given by isomorphisms in the Fukaya $A_\infty$ category of $X$ instead of just Hamiltonian isotopies \cite{FOOO-book, Auroux07, Auroux09, AAK12}. From this viewpoint, when the target manifold is a symplectic quotient, one can as well try to construct the mirror as a moduli space of pairs $(L,\nabla)$ where the equivalence relation is now given by isomorphisms in Woodward's quasimap $A_\infty$ category \cite{Woodward11}. Indeed the modified gluing given by the corresponding wall-crossing formulas would still cancel the nontrivial monodromy of the semi-flat complex structure around the discriminant locus. We call the resulting mirror the {\em quasimap SYZ mirror} for $X$.

In physical terms, such a mirror is precisely what one would get by applying duality on gauged linear sigma models (GLSMs). Thus it is natural to expect that the quasimap SYZ mirrors would coincide with the ones written down by physicists \cite{Leung-Vafa98, Hori-Vafa00, HIV00}, and it would differ from the original SYZ mirror by a mirror transformation, or equivalently, the {\em quantum Kirwan map} \cite{Ziltener14, Woodward_quantumKirwan_I, Woodward_quantumKirwan_II, Woodward_quantumKirwan_III}. We are going to see that this is indeed the case when $X$ is a toric Calabi-Yau manifold.

\section{Woodward's quasimap Floer theory}

Let $G$ be a compact connected Lie group and $\mathfrak{g}$ be its Lie algebra. Let $(\tilde{X}, \tilde{\omega})$ be a Hamiltonian $G$-manifold of real dimension $2m$ with moment map $\mu: \tilde{X} \to \mathfrak{g}^*$. We assume that $(\tilde{X}, \tilde{\omega})$ is {\em aspherical}, meaning that $\int_{S^2} \varphi^*\tilde{\omega} = 0$ for any smooth map $\varphi: S^2 \to \tilde{X}$. Suppose that $G$ acts freely on the level set $\mu^{-1}(0)$ so that the symplectic quotient
$$X := \tilde{X} \sslash G = \mu^{-1}(0) / G$$
is a smooth symplectic manifold equipped with the reduced symplectic structure $\omega := \tilde{\omega}_\text{red}$; more generally, one can relax this condition a bit by assuming that the action has finite stabilizers so that $X$ is a symplectic orbifold.

Let $L \subset X$ be an embedded compact Lagrangian submanifold. Then its preimage $\tilde{L} = \mu^{-1}(L)$ is a {\em $G$-Lagrangian} in $\tilde{X}$, i.e. an embedded $G$-invariant Lagrangian submanifold contained in $\mu^{-1}(0)$. We equip $L$ with a {\em brane structure}, i.e. a $G$-equivariant spin structure and a flat $U(1)$-connection $\nabla$ on $L$, where the gauge equivalence class of $\nabla$ is determined by its holonomy $\exp 2\pi \langle b, \cdot \rangle \in \text{Hom}(H_1(L), U(1)) \cong H^1(L; \bR)/H^1(L, \bZ)$.\footnote{Strictly speaking, due to convergence issues, one should use Novikov coefficients $\Lambda$ instead of complex coefficients, but this technicality will be ignored in this note for simplicity in our exposition and because convergence is not a problem in all our examples.}

In \cite{Woodward11} (see also \cite{Woodward11_corrected}), Woodward developed a {\em quasimap Floer theory} as the zero-area limit of Frauenfelder's {\em gauged Lagrangian Floer theory} \cite{Frauenfelder-PhDthesis, Frauenfelder04}. The latter theory counts pairs $(A, u)$ consisting of a connection $A \in \Omega^1(\Sigma, \mathfrak{g})$ on an open Riemann surface $\Sigma$ and a map $u: \Sigma \to \tilde{X}$ satisfying the {\em vortex equations}
\begin{equation*}
\bar{\partial}_A u = 0,\quad F_A + u^*\mu \text{Vol}_\epsilon = 0,
\end{equation*}
and Lagrangian boundary conditions; here $\text{Vol}_\epsilon = \epsilon \text{Vol}_0$ is a multiple of a fixed area form $\text{Vol}_0$ on $\Sigma$. This can be regarded as an open-string counterpart of the {\em symplectic vortex equations} \cite{CGMS02, Gaio-Salamon05}.

Woodward \cite{Woodward11} observed that the zero-area limit $\epsilon \to 0$ of gauged Lagrangian Floer theory defines a cohomology theory which gives an obstruction to displaceability of Lagrangian submanifolds in the symplectic quotient $X$ that is much more computable than the ordinary Lagrangian Floer cohomology \cite{FOOO-book}. This was applied successfully to the displaceability problem for manifolds and even orbifolds, especially in the toric case \cite{Wilson-Woodward13}. We will not go into the details of Woodward's theory here; instead we refer the interested readers to the original papers \cite{Woodward11, Woodward11_corrected} for details and to e.g. \cite[Appendix A]{Wu-Xu15} for an overview. In the following we will only recall the results we need.

Given a $G$-Lagrangian $\tilde{L}$ equipped with a brane structure $b$, Woodward constructed an $A_\infty$ algebra $QA_\infty(\tilde{L},b)$, called the {\em quasimap $A_\infty$ algebra}, using so-called {\em quasi-disks}:\footnote{In order to define the $A_\infty$ structure, one in fact needs to consider moduli spaces of {\em holomorphic treed quasidisks} which are configurations consisting of gradient flow lines (after choosing a Morse function and a Riemannian metric on $L$) and holomorphic disks in $\tilde{X}$ with boundary on $\tilde{L}$. But for the purpose of this note, we only need to consider holomorphic quasidisks.}
\begin{defn}\label{defn:quasidisk}
Let $\tilde{L}$ be a $G$-Lagrangian in $\tilde{X}$ and $J$ be a $G$-invariant compatible almost complex structure on $\tilde{X}$.
A {\em holomorphic quasidisk} for $L = \mu(\tilde{L})$ is a $J$-holomorphic map $u: D \to \tilde{X}$ from the unit disk $D \subset \bC$ (equipped with the standard complex structure) which maps the boundary $\partial D$ to $\tilde{L}$.
An {\em isomorphism} of quasidisks $u_j: D \to \tilde{X}$, $j = 0, 1$ consists of a biholomorphism $\varphi: D \to D$ and an element $g \in G$ such that $\varphi^* u_1 = g u_0$.
%The moduli space of isomorphism classes of holomorphic quasi-disks in $(X,L)$ is denoted by $\overline{MW}_0(L)$.
\end{defn}

%Let $D_2(X, L) \subset H_2(X, L)$ be the image of the Hurwitz map $\pi_2(X, L) \to H_2(X, L)$. The space of isomorphism classes of holomorphic quasi-disks representing a class $\beta \in D_2(X, L)$ is denoted by $\overline{MW}_0(L, \beta)$.

The definition of quasidisk invariants in \cite{Woodward11} does not involve Kuranishi structures, and thus is considerably simpler than that of open Gromov-Witten invariants \cite{FOOO-book, FOOO-toricI, FOOO-toricII, FOOO-toricIII}. The quasidisk invariants are also much easier to compute since there are no sphere bubbling for quasidisks which are just holomorphic disks in $\tilde{X}$.

\begin{prop}[Proposition 3.7 in \cite{Woodward11}]
\label{prop:weakly_unobstr}
Let $\tilde{L}$ be a $G$-Lagrangian in $\tilde{X}$.
Suppose that $J_0$ is a $G$-invariant compatible almost complex structure on $\tilde{X}$ (satisfying a certain convexity condition \cite[Condition (H3)]{Wu-Xu15}) such that every non-constant stable $J_0$-holomorphic disk in $(\tilde{X}, \tilde{L})$ is regular and has Maslov index at least 2. Then the $A_\infty$ algebra $QA_\infty(\tilde{L},b)$ is weakly unobstructed, i.e. the {\em central charge} $\mathfrak{m}_0(1)$ of $QA_\infty(\tilde{L},b)$ is a multiple of $\mathbf{1}_L \in H^0(L; \bC)$ for any brane structure $b$ on $L$ and is given by the following formula
\begin{equation*}
\mathfrak{m}^b_0(1) = \sum_{[u]: I(u) = 2} q^{-\int_{D} u^*\tilde{\omega}} e^{2\pi \langle b, \partial u\rangle} \textbf{1}_L,
%\mathfrak{m}^b_0(1) = \sum_{\substack{[u] \in \overline{MW}_0(L)\\ I(u) = 2}} q^{-A(u)} e^{\langle b, \partial u\rangle} \textbf{1}_L,
\end{equation*}
where $I(u)$ is the Maslov index of $u$ so that the sum is over isomorphism classes of all Maslov index 2 quasidisks.
\end{prop}

When $QA_\infty(\tilde{L},b)$ is weakly unobstructed, we call
\begin{equation*}
W^\text{QF}(L,b) := \sum_{[u]: I(u) = 2} q^{-\int_{D} u^*\tilde{\omega}} e^{2\pi \langle b, \partial u\rangle}
\end{equation*}
the {\em quasimap Floer superpotential} for $(L,b)$.
In this case, $\mathfrak{m}^b_1 \circ \mathfrak{m}^b_1 = 0$, so that the quasimap Floer cohomology $HQF(\tilde{L},b)$ is well-defined. The cohomology vanishes if $L$ is displaceable but is non-vanishing if $(L,b)$ is a critical point of the function $W^\text{QF}$, thus giving rise to an obstruction to the non-displaceability of $L$ in $X$.

From now on, we will restrict ourselves to the toric case, and we shall recall Woodward's computation in this case. We take $\tilde{X} = \bC^m$ equipped with the standard symplectic structure $\omega_0$. Consider the diagonal action of $T^m$ on $\bC^m$. Let $G \subset T^m$ be a subtorus with moment map $\mu: \tilde{X} \to \mathfrak{g}^*$.
We assume that $G$ acts on $\mu^{-1}(0)$ freely so that the quotient
$$X = \mu^{-1}(0)/G$$
is a toric manifold, equipped with the residual action of $T := T^m / G$ and moment map
$$\phi: X \to \mathfrak{t}^*.$$

The moment map image is given by a convex polyhedron
\begin{equation*}
\Delta := \phi(X) = \{ \mathbf{x} \in \mathfrak{t}^* \mid \ell_i(\mathbf{x}) \geq 0 \},
\end{equation*}
where for each $i = 1, \ldots, m$, $\ell_i(\mathbf{x}) := 2\pi \left( \langle \mathbf{x}, v_i \rangle - \lambda_i \right)$ is the defining linear function for a facet of $\Delta$, the lattice vector $v_i \in N := \bZ^n$ is the (inward) normal to the facet and $\lambda_i \in \bR$ is a constant. We will identify $\mathfrak{t}$ with $N \otimes_\bZ \bR$ and $\mathfrak{t}^*$ with $M \otimes_\bZ \bR$ as vector spaces, where $M = N^\vee = \text{Hom}(N, \bZ)$ is the dual lattice.

For each $\mathbf{x} \in \text{Int}(\Delta)$, the moment map fiber $L_\mathbf{x} := \phi^{-1}(\mathbf{x}) \subset X$ is a Lagrangian torus, whose preimage $\tilde{L}_\mathbf{x} \subset \bC^m$ is a standard torus
\begin{equation*}
\tilde{L}_\mathbf{x} = \{ (X_1, \ldots, X_m) \in \bC^m \mid |X_i|^2 = \ell_i(\mathbf{x})/2\pi \text{ for $i = 1, \ldots, m$} \}.
\end{equation*}
Let $J_0$ be the standard complex structure on $\bC^m$. Then we have the following results (as special cases of the main results in Cho-Oh \cite{Cho-Oh06}).

\begin{prop}[Proposition 6.1 and Corollary 6.2 in \cite{Woodward11}]
Any holomorphic quasidisk in $\tilde{X} = \bC^m$ with boundary in $\tilde{L}_\mathbf{x}$ is given by a Blaschke product
\begin{equation*}
u(z) = \left( \sqrt{\frac{\ell_i(\mathbf{x})}{2\pi}} \prod_{k=1}^{d_i} \frac{z - \alpha_{i,k}}{1 - \overline{\alpha_{i,k}}z} \right)_{i=1}^m.
\end{equation*}
Also, every stable $J_0$-holomorphic disk is regular.
\end{prop}

Furthermore, all quasidisks have Maslov indices at least 2. Combining with Proposition \ref{prop:weakly_unobstr}, we have
\begin{prop}[Corollary 6.4 in \cite{Woodward11}]
The quasimap $A_\infty$ algebra $QA_\infty(\tilde{L}_\mathbf{x},b)$ is weakly unobstructed for any $b \in H^1(\tilde{L}_\mathbf{x}; \bC)$, and the quasimap Floer superpotential $W^\text{QF}: H^1(\tilde{L}_\mathbf{x}; \bC) / H^1(\tilde{L}_\mathbf{x}; \bZ) \to \bC$ is given by
\begin{equation*}
W^\text{QF}(b) = \sum_{i=1}^m e^{2\pi \langle b, v_i\rangle} q^{-\ell_i(\mathbf{x})}.
\end{equation*}
\end{prop}

The function $W^\text{QF}$ coincides with the {\em Givental-Hori-Vafa superpotential} \cite{Givental98, Hori-Vafa00} for the toric manifold $X$.

\section{Quasimap SYZ construction}

\subsection{Toric Calabi-Yau manifolds and their physical mirrors}

We now let $X$ be a toric Calabi-Yau manifold of complex dimension $n$; here by {\em Calabi-Yau} we mean that the canonical line bundle $K_X$ is trivial. Recall that the lattice vectors $v_1, \ldots, v_m \in N$ are in a one-to-one correspondence with the toric prime divisors $D_1, \ldots, D_m \subset X$ respectively, and the canonical divisor of $X$ is given by $-\sum_{i=1}^m D_i$. So $X$ is Calabi-Yau if and only if there exists a lattice vector $u \in M$ such that $\langle u, v_i\rangle = 1$ for $i = 1, \ldots, m$ \cite{CLS_toric_book}.
Alternatively, this is equivalent to the existence of $u\in M$ such that the corresponding character $\chi^u \in \text{Hom}(M \otimes_\bZ \bC^\times, \bC^\times)$ defines a {\em holomorphic} function on $X$ with simple zeros exactly along each of the toric prime divisors $D_i$'s and non-vanishing elsewhere. Note that $X$ is necessarily noncompact in this case.

By choosing a suitable basis of $N \cong \bZ^n$, we may write
$$v_i = (w_i, 1) \in N = \bZ^{n-1} \oplus \bZ,$$
where $w_i \in \bZ^{n-1}$ and $w_m = 0 \in \bZ^{n-1}$. We will also assume that $X$ is {\em semi-projective}, meaning that the natural map $\phi: X \to \text{Spec}(H^0(\mathcal{O}_X))$ is projective; combinatorially this is equivalent to convexity of the support of the fan $\Sigma$ of $X$ \cite[p.332]{CLS_toric_book}. In this case, the toric Calabi-Yau manifold $X$ is a crepant resolution of an affine toric variety (defined by the cone $|\Sigma|$) with Gorenstein canonical singularities; also, $X$ can be presented as a symplectic quotient
$$X = \mu^{-1}(0)/G,$$
where $G \subset T^m$ is a subtorus of dimension $r := m - n$.

An important class of examples of toric Calabi-Yau manifolds is given by total spaces of the canonical line bundles $K_Y$ over compact toric manifolds $Y$.
For example, the total space of $K_{\bP^1}=\mathcal{O}_{\bP^1}(-2)$ is a toric Calabi-Yau surface whose fan $\Sigma$ has rays spanned by the lattice vectors
$$v_1=(1,1), v_2=(0,1), v_2=(-1,1,1), v_3=(-1,1)\in N = \bZ^2.$$
Another example is given by the total space of $K_{\bP^2}=\mathcal{O}_{\bP^2}(-3)$, which is a toric Calabi-Yau 3-fold whose fan $\Sigma$ has rays spanned by the lattice vectors
$$v_1=(1,0,1), v_2=(0,1,1), v_3=(-1,-1,1), v_4=(0,0,1)\in N = \bZ^3.$$

Mirror symmetry in this setting is known as {\em local mirror symmetry} because it originated from an application of mirror symmetry techniques to Fano surfaces (e.g $\bP^2$) lying inside compact Calabi-Yau manifolds and could be derived via physical arguments from mirror symmetry for compact Calabi-Yau hypersurfaces in toric varieties by taking certain limits in the complexified K\"ahler and complex moduli spaces \cite{KKV97}.

The mirror of a toric Calabi-Yau manifold $X$ is predicted to be a family of affine hypersurfaces in $\bC^2 \times (\bC^\times)^{n-1}$ \cite{Leung-Vafa98, CKYZ99, Hori-Vafa00, HIV00} explicitly written as
\begin{equation}\label{eqn:toricCY_mirror}
\check{X}_t = \left\{ (u, v, z_1, \ldots, z_{n-1}) \in \bC^2 \times (\bC^\times)^{n-1} \mid uv = \sum_{i=1}^{m} \check{C}_i z^{w_i} \right\},
\end{equation}
where the coefficients $\check{C}_i \in \bC$ are constants (without loss of generality, we will set $\check{C}_m = 1$) subject to the constraints
$$t_a = \prod_{i=1}^m \check{C}_i^{D_i\cdot \gamma_a},\quad a=1,\ldots,r;$$
here $t = (t_1, \ldots, t_r)$ are coordinates on the mirror complex moduli $\check{\mathcal{M}}_B := \mathbb{K}^\vee \otimes_{\bZ} \bC^\times \cong (\bC^\times)^r$. $\check{X}_t$ is Calabi-Yau since
\begin{equation*}
\check{\Omega}_t := \text{Res}\left[ \frac{du \wedge dv \wedge d\log z_1 \wedge \cdots \wedge d\log z_{n-1}}{uv - \sum_{i=1}^{m} \check{C}_i z^{w_i}} \right]
\end{equation*}
is a nowhere vanishing holomorphic volume form on $\check{X}_t$.

The mirror of $X = K_{\bP^1}$ is given by
\begin{align*}
\check{X}_t = \left\{ (u,v,z) \in \bC^2 \times \bC^\times \mid uv = 1 + z + \frac{t}{z} \right\},
\end{align*}
while the mirror of $X = K_{\bP^2}$ is given by
\begin{align*}
\check{X}_t = \left\{ (u,v,z_1,z_2) \in \bC^2 \times (\bC^\times)^2 \mid uv = 1 + z_1 + z_2 + \frac{t}{z_1z_2} \right\};
\end{align*}
here $t$ is a coordinate on the mirror complex moduli $\check{\mathcal{M}}_B \cong \bC^\times$ in both examples.

\subsection{Constructing mirrors by SYZ}

As we have mentioned in the introduction, the SYZ proposal \cite{SYZ96} suggests a way to construct a mirror (as a complex manifold) for a given Calabi-Yau manifold $X$ (regarded as a symplectic manifold), namely, by fiberwise dualizing a Lagrangian torus fibration $\rho:X \to B$ on $(X,\omega)$. In general this does not give the correct mirror due to the existence of singular fibers in $\rho$, which prevent the natural semi-flat complex structure $\check{J}_0$ on the total space of the dual fibration from extending across the singular fibers; another way to formulate this problem is to say that $\check{J}_0$ has nontrivial monodromy around the discriminant locus coming from nontrivial monodromy of the integral affine structure on the smooth locus $B_0 \subset B$.

In the original SYZ paper \cite{SYZ96}, it was suggested that the genuine complex structure should be given by a deformation of $\check{J}_0$ by nontrivial instanton corrections coming from holomorphic disk counting invariants. In the toric Calabi-Yau case, such instanton corrections are given precisely by the genus 0 {\em open Gromov-Witten invariants}. The key observation in this note is that, for the purpose of just cancelling the nontrivial monodromy, one can as well use quasidisk invariants instead of genus 0 open Gromov-Witten invariants. This leads to the mirror construction described as follows.

We start with a Lagrangian torus fibration
$$\rho: X \to B$$
on $(X, \omega)$. We assume that this fibration comes from the fiberwise quotient of a fibration on $\tilde{X}$; more precisely, we assume that there exists a $G$-Lagrangian torus fibration
$$\tilde{\rho}: \tilde{X} \to \tilde{B}$$
on $(\tilde{X}, \tilde{\omega})$ such that
\begin{itemize}
\item[$\bullet$]
$B$ sits inside $\tilde{B}$ as an affine submanifold, and
\item[$\bullet$]
$\rho$ is given by the fiberwise quotient of the restriction of $\tilde{\rho}$ to $\mu^{-1}(0)$ by $G$.
\end{itemize}
The {\em quasimap SYZ mirror construction} then proceeds as follows (mimicking the ordinary SYZ mirror construction \cite{Auroux07, Auroux09, AAK12, CLL12, CCLT13}):
\begin{itemize}
\item[Step 1]
Over the smooth locus $B_0 := B \setminus \left( \partial B \cup \Gamma \right)$, the pre-image $X_0 := \rho^{-1}(B_0)$ can be identified with the quotient $T^*B_0/\Lambda^\vee$ by Duistermaat's action-angle coordinates \cite{Duistermaat80}.

\item[Step 2]
Define the {\em semi-flat} mirror $\check{X}_0$ as $TB_0/\Lambda$, which is not the correct mirror because the natural semi-flat complex structure $\check{J}_0$ on $\check{X}_0$ {\em cannot} be extended further to {\em any} (partial) compactification of $\check{X}_0$ due to nontrivial monodromy of the integral affine structure on $B_0$ around the discriminant locus $\Gamma$.

\item[Step 3]
Obtain the correct and (partially) compactified mirror $\check{X} \supset \check{X}_0$ by modifying the gluing of complex charts in $\check{X}_0$ using the wall-crossing formulas for the counting of {\em quasi-disks} bounded by fibers of $\rho$ (or more correctly, their lifts as fibers of $\tilde{\rho}$).
\end{itemize}

In the case of the usual SYZ construction where one uses Lagrangian Floer theory, such a procedure was first pioneered by Auroux \cite{Auroux07, Auroux09} where he treated the first nontrivial example of toric Calabi-Yau manifolds, namely, for $X = \bC^n$. His results were later generalized to all (semi-projective) toric Calabi-Yau manifolds in \cite{CLL12} and orbifolds in \cite{CCLT13}, and also certain blowups of toric varieties in \cite{AAK12}.

The SYZ mirror constructed in this way can be rigorously defined as a moduli space of objects in the Fukaya $A_\infty$ category; see \cite[Appendix A]{AAK12} for a nice and detailed explanation. We expect that the quasimap SYZ mirror has a similar interpretation using Woodward's quasimap $A_\infty$ category.

\subsection{The Gross fibration}

Here we recall the construction of the Gross fibration on a toric Calabi-Yau manifold $X$.\footnote{Such fibrations were in fact first independently constructed by Gross \cite{Gross01} and Goldstein \cite{Goldstein01}, but we prefer the term ``Gross fibration'' because Gross did a detailed analysis of the discriminant loci of the fibrations and constructed such fibrations mainly for the purpose of understanding SYZ mirror symmetry.}

To begin with, recall that the lattice vector $u \in M \subset \mathfrak{t}^*$, which defines the hyperplane containing all the ray generators $v_i$'s, corresponds to a holomorphic function $\chi^u: X \to \bC$ with simple zeros along each toric prime divisor $D_i \subset X$. We equip $X$ with a toric K\"ahler structure $\omega$ and consider the action by the subtorus $T_0 \subset T \cong T^n$ which preserves $\chi^u$, or equivalently, the subtorus whose action preserves the canonical holomorphic volume form $\Omega$ on $X$. Let $\rho_0: X \to \bR^{n-1}$ be the corresponding moment map which is given by composing the $T$-moment map with the projection along the ray in $\mathfrak{t}^*$ spanned by $u$.
\begin{prop}[Goldstein \cite{Goldstein01}, Gross \cite{Gross01}]
For any nonzero constant $\epsilon \in \bC^\times$, the map defined by
\begin{align*}
\rho := \left( \rho_0, |\chi^u - \epsilon| \right): X \to B:= \bR^{n-1} \times \bR_{\geq0},
\end{align*}
is a special Lagrangian torus fibration, where the fibers are special with respect to the meromorphic volume form
$$\Omega_\epsilon := \frac{\Omega}{\chi^u - \epsilon}.$$
\end{prop}

We call $\rho$ the {\em Gross fibration}, which is {\em non-toric} in the sense that its fibers are not invariant under the $T$-action. Its discriminant locus can be explicitly described, namely, a fiber of $\rho$ is singular if and only if either
\begin{itemize}
\item
it intersects nontrivially with (and hence is contained inside) the hypersurface $D_\epsilon \subset X$ defined by $\chi^u = \epsilon$, in which case the fiber is mapped to a point on the boundary $\partial B = \bR^{n-1} \times \{0\}$, or
\item
it contains a point where the $T_0$-action is not free, i.e. when at least two of the homogeneous coordinates on $X$ vanish, in which case the fiber is mapped to the image $\Gamma$ of the codimension 2 subvariety
$$\bigcup_{i\neq j} D_i \cap D_j$$
under $\rho$.
\end{itemize}
We regard $B$ as a (tropical) affine manifold with boundary $\partial B$ and singularities $\Gamma$. Note that $\Gamma$ is a real codimension 2 subset in $B$.

By definition, the {\em wall(s)} in the base of a Lagrangian torus fibration is the loci of smooth fibers which bound nonconstant {\em Maslov index 0} holomorphic disks in $X$.
For the Gross fibration on a toric Calabi-Yau manifold, there is a unique wall given by the hyperplane
$$H := \bR^{n-1} \times \{|\epsilon|\} \subset B,$$
which is parallel to the boundary $\partial B$.
The wall $H$ contains the discriminant locus $\Gamma$ as a {\em tropical hypersurface}, and it divides the base $B$ into two chambers:
\begin{align*}
B_+  := \bR^{n-1} \times (|\epsilon|, +\infty),
B_-  := \bR^{n-1} \times (0,|\epsilon|)
\end{align*}
The Lagrangian torus fibers over $B_+$ and $B_-$ behave differently in a Floer-theoretic sense, and this leads to nontrivial wall-crossing formulas which were used to construct the SYZ mirror for $X$ \cite{Auroux07, Auroux09, AAK12, CLL12, CCLT13}.

For example, the base $B$ of the Gross fibration on $X = K_{\bP^2}$ is an upper half space in $\bR^3$, and the discriminant locus is a graph which is contained in a hyperplane $H$ parallel to the boundary $\partial B$, as described in Figure \ref{fig:KP2_base}.

\begin{figure}
\begin{tikzpicture}
\draw[dashed] (4,3) -- (14,1.5);
\draw[dashed] (4,3) -- (10,5);
\draw[thick] (4.1,3) -- (3.9,3) node[left] {$|\epsilon|$};
\draw (3.9,1.5) node[left] {$B_-$};
\draw (3.9,4.5) node[left] {$B_+$};
\draw[->] (3,0.15) -- (14,-1.5);
\draw[->] (3.1,-0.3) -- (10,2);
\draw[->] (4,-1) -- (4,6);
\draw [thick, red] (6,3.2) -- (9,3.5);
\draw [thick, red] (9,3.5) -- (10.3,3.933);
\draw [thick, red] (9,3.5) -- (10.5,3.275);
\draw [thick, red] (10.3,3.933) -- (10.5,3.275);
\draw [thick, red] (10.3,3.933) -- (10.9,4.7);
\draw [thick, red] (10.5,3.275) -- (12.3,2.2);
\end{tikzpicture}
\caption{The base of the Gross fibration for $X = K_{\bP^2}$}\label{fig:KP2_base}
\end{figure}
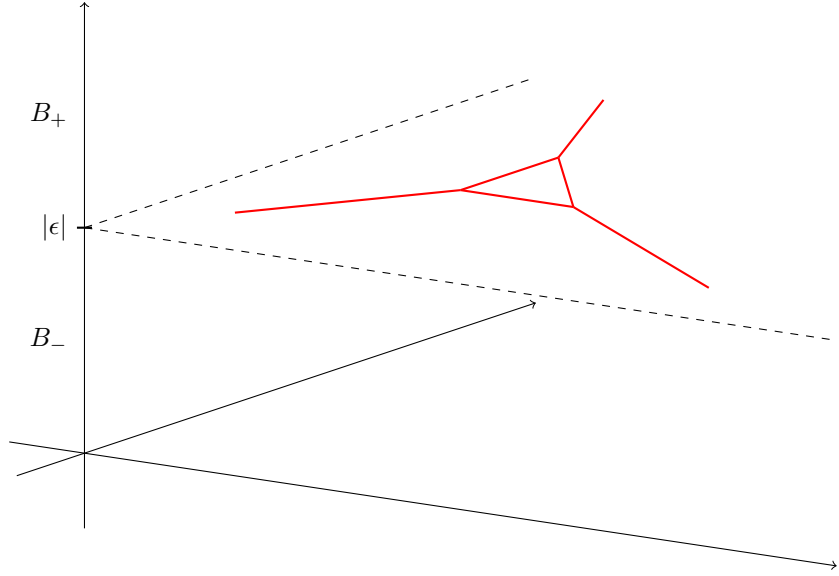

For our purpose, a simple but key observation is that the Gross fibration satisfies all the assumptions in the previous subsection. Indeed, it is not hard to see that $\rho$ is nothing but the fiberwise quotient of the Harvey-Lawson fibration on $\bC^m$ \cite{Harvey-Lawson82}:
\begin{equation}\label{eqn:Harvey-Lawson fibration}
\begin{split}
\tilde{\rho}: \bC^m & \to \tilde{B} := \bR^{m-1} \times \bR_{\geq 0},\\
\left( X_1, \ldots, X_m \right) & \mapsto \left( |X_1|^2 - |X_m|^2, \ldots, |X_{m-1}|^2 - |X_m|^2, |X_1 X_2 \cdots X_m - \epsilon| \right).
\end{split}
\end{equation}

\begin{lem}
The Gross fibration $\rho:X \to B$ is given by the fiberwise quotient of $\tilde{\rho}|_{\mu^{-1}(0)}$ by the $r$-dimensional subtorus $G \subset T^m$.
\end{lem}
\begin{proof}
This follows by noting that the preimage of a fiber of the $T$-moment map is a standard torus in $\bC^m$, and that the holomorphic function $\chi^u: X \to \bC$ is lifted to the monomial $X_1 X_2 \cdots X_m: \bC^m \to \bC$.
\end{proof}

The embedding of affine manifolds $B \hookrightarrow \tilde{B}$ can be explicitly seen as follows. Recall that there is an exact sequence
$$ 0 \to \mathfrak{g} \to \bR^m \to \mathfrak{t} \to 0. $$
Dualizing, we have
$$ 0 \to \mathfrak{t}^* \to \left(\bR^m\right)^* \to \mathfrak{g}^* \to 0,$$
where the first map
$$\mathfrak{t}^* \hookrightarrow \left(\bR^m\right)^*$$
is defined by
$$ \mathbf{x} \mapsto \left( \ell_1(\mathbf{x}), \ldots, \ell_m(\mathbf{x}) \right).$$
Let $\mathfrak{t}_0$ be the Lie algebra of the subtorus $T_0 \subset T$ which preserves the holomorphic volume form $\Omega$ on $X$, and let $\mathfrak{R}_0$ denote the Lie algebra of the subtorus $(T^m)_0 \subset T^m$ which preserves the holomorphic volume form $dX_1 \wedge \cdots \wedge dX_m$ on $\bC^m$. Then the above map induces an embedding
$$ \mathfrak{t}_0^* \hookrightarrow \mathfrak{R}_0^* $$
which in turn defines the embedding
\begin{equation}\label{eqn:embedding_base}
B \hookrightarrow \tilde{B}.
\end{equation}

From this we can see that the fiberwise quotient respects the wall and chamber structures, namely, fibers over the wall $H$ (resp. the chambers $B_+$ and $B_-$) in the base of the Gross fibration are exactly quotients by $G$ of fibers over the wall $H$ (resp. the chambers $B_+$ and $B_-$) in the base of the Harvey-Lawson fibration.

For instance, we may consider $X = K_{\bP^1} = \bC^3 \sslash S^1$. The embedding \eqref{eqn:embedding_base} in this example is shown in Figure \ref{fig:C3andKP1_base}.

\begin{figure}
\begin{tikzpicture}
\draw [fill, black!10] (6.1103, 0.7034) -- (6.1103, 6.7034) -- (14, 5.52) -- (14, -0.48) -- (6.1103, 0.7034);
\draw[dashed] (4,3) -- (14,1.5);
\draw[dashed] (4,3) -- (10,5);
\draw[thick] (4.1,3) -- (3.9,3) node[left] {$|\epsilon|$};
\draw (3.9,1.5) node[left] {$B_-$};
\draw (3.9,4.5) node[left] {$B_+$};
\draw[->] (3,0.15) -- (14,-1.5);
\draw[->] (3.1,-0.3) -- (10,2);
\draw[->] (4,-1) -- (4,6);
\draw [thick, red] (6,3.2) -- (8.08,3.408);
\draw [dashed, thick, red] (8.08,3.408) -- (10,3.6);
\draw [dashed, thick, red] (10,3.6) -- (10.9, 4.6);
\draw [dashed, thick, red] (10,3.6) -- (11.0733, 2.959);
\draw [thick, red] (11.0733, 2.959) -- (12.4, 2.1667);
\draw[->] (6.1103, 0.7034) -- (6.1103, 6.7034);
\draw[->] (6.1103, 0.7034) -- (14,-0.48);
%\draw [dashed, black!40] (5.2, 3.84) -- (13.8, 2.55);
\draw [dashed, black!40] (6.1103, 3.7035) -- (13.8, 2.55);
\draw [fill, blue] (8.08,3.408) circle [radius=0.075]; %node[below left, black] {$s_1$};
\draw [fill, blue] (11.0733,2.959) circle [radius=0.075];
%\draw [thick, red] (9,3.5) -- (10.3,3.933);
%\draw [thick, red] (9,3.5) -- (10.5,3.275);
%\draw [thick, red] (10.3,3.933) -- (10.5,3.275);
%\draw [thick, red] (10.3,3.933) -- (10.9,4.7);
%\draw [thick, red] (10.5,3.275) -- (12.3,2.2);
\end{tikzpicture}
\caption{The base of the Gross fibration for $K_{\bP^1}$ contained inside the base of the Harvey-Lawson fibration for $\bC^3$}\label{fig:C3andKP1_base}
\end{figure}
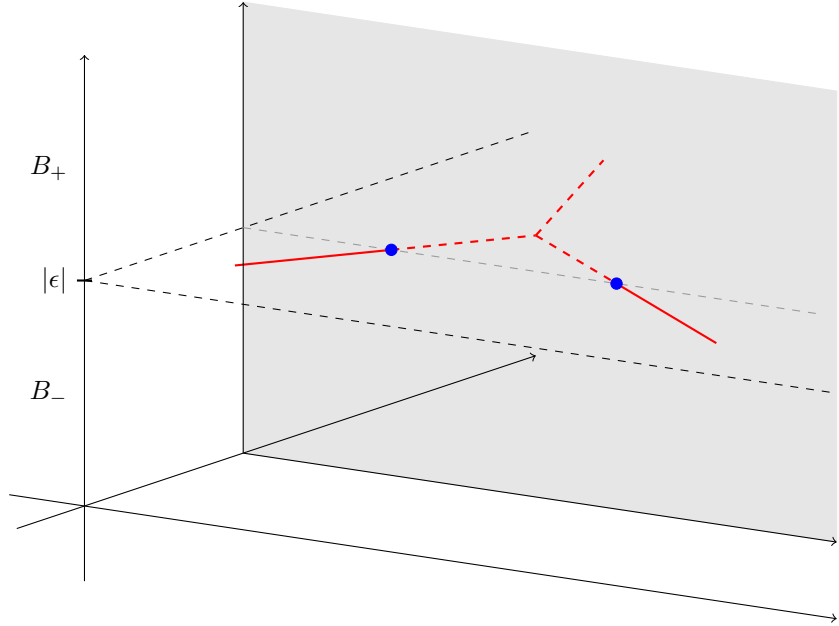

\subsection{Wall-crossing for quasidisk invariants and the SYZ mirror}

Just as in the case of Lagrangian Floer theory \cite{Auroux07, Auroux09, AAK12, CLL12, CCLT13}, when fibers of $\tilde{\rho}$ over different chambers of $B$ are identified by a wall-crossing gluing map, the quasimap Floer superpotentials are identified. This is precisely how we compute the instanton corrections and construct the SYZ mirror.

We consider the number of {\em Maslov index 2} stable quasidisks for the Lagrangian fibers of $\rho$. Geometrically, Maslov index 2 means that the stable quasidisks intersect with the hypersurface $\tilde{D}_\epsilon$, which is defined as the zero set of $X_1X_2\cdots X_m - \epsilon$, in $\bC^m$ at only one point with multiplicity one. As one moves from one chamber to the other by crossing the wall $H$, the number of Maslov index 2 quasidisks bounded by the corresponding Lagrangian torus fiber jumps, exhibiting a {\em wall-crossing phenomenon}.

\begin{lem}
For any Lagrangian torus fiber $\tilde{L}$ over the chambers $B_+$ and $B_-$, every non-constant stable holomorphic disk in $(\tilde{X}, \tilde{L})$ has positive Maslov index and is regular under the standard complex structure $J_0$.
\end{lem}
\begin{proof}
Fibers over $B_+$ are Hamiltonian isotopic to moment map fibers which are standard tori in $\bC^m$, so this follows from \cite[Corollary 6.2]{Woodward11} which in turn follows from Cho-Oh's classification results \cite{Cho-Oh06} and an induction argument as in \cite[Theorem 11.1]{FOOO-toricI} (see also \cite[Corollary 6.2]{Woodward11}).

Fibers over $B_-$ are Hamiltonian isotopic to the Chekanov tori \cite{Auroux09} in $\bC^m$, so we can apply the classification in \cite{Auroux09} or \cite[Lemma 4.31]{CLL12} which says that there is a unique quasidisk, and then apply induction again for proving the regularity for stable disks, just as in \cite[Corollary 6.2]{Woodward11}.
\end{proof}

\begin{prop}
The central charge $\mathfrak{m}_0(1)$ of the quasimap $A_\infty$ algebra is given by
\begin{align*}
\mathfrak{m}^b_0(1) =
\left\{\begin{array}{cc}
\sum_{i=1}^m \check{C}_i v^{-1}z^{w_i} & \text{if $L$ is over $B_+$,}\\
u & \text{if $L$ is over $B_-$.}
\end{array}
\right.
\end{align*}
\end{prop}
\begin{proof}
The lifts of fibers over $B_+$ of Clifford type and they are Hamiltonian isotopic to standard tori in $\bC^m$, so they bound as many disks as a standard torus in $\bC^m$. In this case, the formula follows from the classification results of Cho-Oh \cite[Theorem 5.2]{Cho-Oh06} and their area formula \cite[Theorem 8.1]{Cho-Oh06} which gives $e^{2\pi \langle b, v_i\rangle} q^{-\ell_i(\mathbf{x})} = \check{C}_i v^{-1}z^{w_i}$ for $i = 1, \ldots, m$.
For fibers over $B_-$, their lifts are of Chekanov type, so they bound only one (family of) disks as shown in \cite[Example 3.3.1]{Auroux09} or \cite[Lemma 4.31]{CLL12}.
\end{proof}

The resulting {\em wall-crossing formula}:
$$ u = v^{-1} \sum_{i=1}^m \check{C}_i z^{w_i}$$
is exactly what we need in order to get the correct mirror. More precisely, we modify the gluing between the complex charts $\check{X}_+ = TB_+/\Lambda$ and $\check{X}_- = TB_-/\Lambda$ using the wall-crossing formula. This cancels the nontrivial monodromy of the complex structure around the discriminant locus $\Gamma \subset B$ and produces the following quasimap SYZ mirror:

\begin{thm}
The quasimap SYZ mirror for the toric Calabi-Yau manifold $X$ is given by the family of affine hypersurfaces
\begin{equation}\label{eqn:SYZmirror_toricCY}
\check{X}_q = \left\{(u,v,z_1,\ldots,z_{n-1}) \in \bC^2 \times (\bC^\times)^{n-1} \mid uv = \sum_{i=1}^{m} \check{C}_i z^{w_i} \right\},
\end{equation}
where the coefficients $\check{C}_i \in \bC$ are constants (with $\check{C}_m = 1$) subject to the constraints
$$q_a = \prod_{i=1}^m C_i^{D_i\cdot \gamma_a},\quad a=1,\ldots,r;$$
here $z^w$ denotes the monomial $z_1^{w^1}\ldots z_{n-1}^{w^{n-1}}$ for $w = (w^1,\ldots,w^{n-1})\in\bZ^{n-1}$,
$q^d$ denotes $\exp\left(-\int_d \omega_\bC\right)$ which can be expressed in terms of the complexified K\"ahler parameters $q_1, \ldots, q_r$, and $\beta_1, \ldots, \beta_m \in \pi_2(X,L)$ are the basic disk classes as before.
\end{thm}

Notice that the quasimap SYZ mirror family \eqref{eqn:SYZmirror_toricCY} is entirely written in terms of symplectic-geometric information such as complexified K\"ahler parameters and quasidisk invariants of $X$, and it coincides with the mirror \eqref{eqn:toricCY_mirror} predicted by physical arguments \cite{Leung-Vafa98, CKYZ99, HIV00}.

\begin{remark}
Strictly speaking, this is the SYZ mirror for the complement of a hypersurface $D_\epsilon$ (zero set of the function $\chi^u - \epsilon$) in $X$ only; the SYZ mirror of $X$ itself should be given by the Landau-Ginzburg model $(\check{X}, W)$ where the superpotential is the function $W := u$.
\end{remark}

\begin{remark}
As in \cite{Woodward11, Wilson-Woodward13}, all the computations of quasidisk invariants and constructions above can be generalized (in a straightforward way) to toric Calabi-Yau orbifolds.
\end{remark}

On the other hand, the SYZ mirror for $\bC^m$ with respect to the Harvey-Lawson fibration \eqref{eqn:Harvey-Lawson fibration} is given by
\begin{equation*}
(\bC^m)^\vee = \left\{ (u,v,Z_1,\ldots,Z_{m-1}) \in \bC^2 \times (\bC^\times)^{m-1} \mid uv = 1 + Z_1 + \cdots + Z_{m-1} \right\}.
\end{equation*}
Notice that $\check{X}_q$ embeds into $(\bC^m)^\vee$ by
$$ u = u, v = v, Z_i = \check{C}_i z^{w_i} \text{ for $i = 1, \ldots, m-1$} $$
via the embedding \eqref{eqn:embedding_base}; this is mirror to the fact that $X$ is a quotient of $\bC^m$, very much like the case of compact toric manifolds as shown in \cite{Chan-Leung10a, Chan-Leung10b}.

\begin{example}\label{eg:SYZ_KP2}
The SYZ mirror of $X = K_{\bP^2}$ is given by
\begin{equation}\label{eqn:mirror_KP2}
\check{X} = \left\{(u,v,z_1,z_2)\in\bC^2\times(\bC^\times)^2 \mid uv = 1 + z_1 + z_2 + \frac{q}{z_1z_2} \right\},
\end{equation}
where $q$ is the K\"ahler parameter which measures the symplectic area of a projective line contained inside the zero section of $K_{\bP^2}$ over $\bP^2$. It embeds into the SYZ mirror of $\bC^4$:
\begin{equation*}
(\bC^4)^\vee = \left\{ (u,v,Z_1,\ldots,Z_3) \in \bC^2 \times (\bC^\times)^3 \mid uv = 1 + Z_1 + Z_2 + Z_3 \right\}
\end{equation*}
as the hypersurface defined by $Z_1 Z_2 Z_3 = q$.
\end{example}

\section{Discussions}

\subsection{}

To deform the semi-flat complex structure $\check{J}_0$ on the mirror $\check{X}_0$ so that it can be extended across the singular fibers, we only need to cancel its nontrivial monodromy around the discriminant locus $\Gamma$. But this condition is {\em not} sufficient to uniquely determine the genus 0 open Gromov-Witten invariants; indeed we are exploiting this flexibility in order to use quasidisk invariants instead of open Gromov-Witten invariants to implement the SYZ mirror construction.

If one imposes further the normalization condition proposed by Gross-Siebert in their program \cite{Gross-Siebert-reconstruction}, which is equivalently to asking that the mirror be written in canonical coordinates \cite{Ruddat-Siebert14}, then the instanton corrections are uniquely determined and can be shown to be precisely given by genus 0 open Gromov-Witten invariants \cite{Lau14}. This leads to the following question: could one give a geometric explanation for why the open Gromov-Witten potentials satisfy the normalization condition?

Fukaya-Oh-Ohta-Ono \cite{FOOO-toricIII} was able to obtain a Frobenius manifold structure on the total cohomology of a compact toric manifold using bulk-deformed genus 0 open Gromov-Witten invariants, and they proved that this is isomorphic to the B-model Frobenius manifold coming from Saito's theory of singularities \cite{Saito83}. It is natural to ask if the Frobenius structure on the total cohomology is more or less unique. If this is the case, then an analogous construction in the case of toric Calabi-Yau manifolds would give an answer to the above question.

\subsection{}

In \cite{CLT11, CCLT13}, it was proved that the so-called {\em SYZ map}, which is defined in terms of generating functions of genus 0 open Gromov-Witten invariants, coincides with the inverse of the toric mirror map for any semi-projective toric Calabi-Yau manifold. The analogue of the SYZ map in the quasimap setting would just be the identity map $t(q) = q$. This is expected because the quasimap SYZ mirror is identical to the physical mirror, which differs from the usual SYZ mirror by a mirror map. We believe that there is a family of SYZ mirrors (and hence a family of SYZ maps) interpolating between the quasimap SYZ mirror and Lagrangian Floer SYZ mirror, which can be described as follows.

What we need is an open-string version of the (genus 0) moduli spaces constructed by Venugopalan in \cite{Venugopalan13}. Her theory is a symplectic version of the theory of stable quasimaps due to Ciocan-Fontanine, Kim and Maulik \cite{Ciocan-Fontanine-Kim-Maulik14}. More precisely, Venugopalan considers the space of finite energy vortices defined on Riemann surfaces obtained from nodal curves with infinite cylinders in the places of nodal and marked points. She showed that this space can be compactified by stable vortices which incorporate both breaking of cylinders and sphere bubbling in the fibers, and she proved that the compactified space is homeomorphic to the corresponding moduli space of stable quasimaps defined in \cite{Ciocan-Fontanine-Kim-Maulik14}.

It is natural to expect that the SYZ construction can still be implemented using an open-string analogue of Venugopalan's invariants (in genus 0) in place of Fukaya-Oh-Ohta-Ono's genus 0 open Gromov-Witten invariants. This would produce a family of SYZ mirrors $\check{X}_\epsilon$ and define an $\epsilon$-SYZ map.

\begin{conj}
The $\epsilon$-SYZ map coincides with the inverse of the $\epsilon$-mirror map.
\end{conj}

Here, the $\epsilon$-mirror map should be given by the $1/z$-coefficient of the $H^2(X)$-part of the function $I - J_{\text{sm}}^\epsilon$, where $J_{\text{sm}}^\epsilon$ is the small $\epsilon$-J-function defined in \cite{Ciocan-Fontanine-Kim14} using the moduli space of genus 0 stable quasimaps. In general, $J_{\text{sm}}^\epsilon$ is a truncation of the classical $I$-function. When $\epsilon = 0$, $J_{\text{sm}}^0 = I$ is the $I$-function itself so that the $\epsilon$-mirror map is nothing but the identity map. When $\epsilon \to \infty$, the $\epsilon$-mirror map is the usual mirror map, and the above conjecture reduces to the open mirror theorem established in \cite{CLT11, CCLT13}.

\begin{remark}
Note that we are not proposing to use the open-string analogue of gauged Gromov-Witten theory, i.e. counting of solutions of the symplectic vortex equations \cite{CGMS02, Gaio-Salamon05} (for a survey on this theory and its applications, see e.g. \cite{Gonzalez16}) because there is no wall-crossing in gauged Gromov-Witten theory as the parameter $\epsilon$ moves and all the information is captured by the quantum Kirwan map \cite{Ziltener14, Woodward_quantumKirwan_I, Woodward_quantumKirwan_II, Woodward_quantumKirwan_III}. In contrast, there is nontrivial wall-crossing phenomenon in the quasimap theory of Ciocan-Fontanine, Kim and Maulik \cite{Ciocan-Fontanine-Kim-Maulik14} as $\epsilon$ moves, and this is what we need if we want to have a nontrivial interpolation between the two extreme SYZ mirrors. In fact it is not known how the above two theories are related to each other.
\end{remark}

\section*{Acknowledgment}
I would like to thank Bumsig Kim for asking what the quasimap analogue of the SYZ mirror would be like. Thanks are also due to Eduardo Gonz\'alez, Ziming Nikolas Ma, Chris Woodward and Guangbo Xu for useful discussions and comments.

This is a write-up of the author's invited talk at the 7th International Congress of Chinese Mathematicians (ICCM) held in August 2016 in Beijing which was jointly hosted by the Academy of Mathematics and Systems Science (AMSS) and the Morningside Center of Mathematics (MCM) of the Chinese Academy of Sciences (CAS). I am grateful to the hosting institutions and organizers for invitation and hospitality.

This research was substantially supported by grants from the Research Grants Council of the Hong Kong Special Administrative Region, China (Project No. CUHK14300314, CUHK14302015 $\&$ CUHK14314516).

\bibliographystyle{amsplain}
\bibliography{geometry}

\end{document}